\documentclass[12pt]{amsart}
\usepackage{amssymb,amsmath,amsthm}
\numberwithin{equation}{section}

\usepackage[shortlabels]{enumitem}
\usepackage{mathrsfs}

\usepackage{color}

\usepackage{url}

\usepackage[pagebackref,colorlinks,linkcolor=red,citecolor=blue,urlcolor=blue,hypertexnames=true]{hyperref}

\newtheorem{theorem}{Theorem}[section]
\newtheorem{lemma}[theorem]{Lemma}

\newtheorem{proposition}[theorem]{Proposition}
\theoremstyle{definition}
\newtheorem{remark}[theorem]{Remark}

\topmargin = - 0.0 in
\newcommand{\df}[1]{{\it{#1}}{\index{#1}}}
\newcommand\cD{\mathcal D}
\newcommand\cB{\mathcal{B}}
\newcommand\C{\mathbb{C}}

\newcommand{\Langle}{\mathop{<}\!}
\newcommand{\Rangle}{\!\mathop{>}}
\newcommand{\xx}{\!\Langle x\Rangle}

\newcommand{\fP}{\mathfrak{P}}
\newcommand{\CC}{\mathbb{C}}
\newcommand{\cL}{\mathscr{L}}

\newcommand\cdotb{\boldsymbol{\cdot}}

\newcommand{\hfP}{\cB_R}

\newcommand{\fL}{\mathfrak{L}}

\newcommand{\vg}{{\tt{g}}}

\newcommand{\gG}{G}

\newcommand{\cF}{\mathcal{F}}

\newcommand{\tb}{\mathfrak{b}}
\newcommand{\ty}{\mathfrak{y}}
\newcommand{\EE}{\mathbb{E}}
\newcommand{\sS}{\mathscr{S}}
\newcommand{\sT}{\mathscr{T}}
\makeindex

\title[Reinhardt Spectrahedra]{Reinhardt Free Spectrahedra}

\author[Ben Jemaa]{Munir Ben Jemaa${}^*$} 
\address{%Department of Mathematics\\
  University of Flordia \\ Gainesville, FL}
\email{munirbenjemaa@ufl.edu} 
\thanks{${}*$ Research supported by the University of Florida 
University Scholars undergraduate research program}

\author[McCullough]{Scott McCullough}
\address{Department of Mathematics\\
  University of Flordia \\ Gainesville, FL}
\email{sam@ufl.edu}


 \numberwithin{equation}{section}

\begin{document}

\maketitle

%\today

\begin{abstract}
 The automorphism group of a particular free spectrahedron 
 is determined via
 a novel argument involving  algebraic methods.
\end{abstract}

\section{Introduction}
\thispagestyle{empty}
 Fix norm $1$ matrices $C_1,C_2$ of size $s\times s.$ 
 For positive integers $n,$ let $M_n(\C)^2$ denote
 the set of pairs $X=(X_1,X_2)$ of $n\times n$ matrices
and let \df{$\fP[n]$} denote those $X\in M_n(\C)^2$ 
 for which the hermitian block $4\times 4$ matrix  \index{$\cL$}
\[
 \cL(X)= \begin{pmatrix} I_s\otimes I_n & C_1\otimes X_1 & C_2\otimes X_2 & 0 \\
   (C_1\otimes X_1)^* & I_s\otimes I_n & 0 & C_2\otimes X_2 \\
  (C_2\otimes X_2)^* & 0 & I_s\otimes I_n & C_1\otimes X_1\\0& (C_2\otimes X_2)^* &
  (C_1\otimes X_1)^* & I_s\otimes I_n \end{pmatrix}
\]
 is positive definite.
 Here  $X_j^*$ is the adjoint (complex transpose)
 of $X_j$ and $I_n$ is the $n\times n$ identity matrix.
 The sequence of sets $\fP= (\fP[n])_n$ is an example
 of a \df{free spectrahedron}.

 Given  $X\in M_n(\CC)^2$ and $Y\in M_m(\CC)^2,$ and 
 a unitary matrix $U\in M_n(\C),$  let \index{$X\oplus Y$}
\[
 X\oplus Y = \left ( \begin{pmatrix} X_1 &0\\0&Y_1\end{pmatrix},
  \begin{pmatrix} X_2 & 0\\ 0& Y_2\end{pmatrix} \right )
\]
 and \index{$U^*XU$}
\[
 U^*XU  =(U^*X_1U, \, U^*X_2U).
\]
 Observe, if $X\in \fP[n]$ and $Y\in \fP[m],$ then 
  $X\oplus Y\in \fP[n+m]$ and $U^*XU \in \fP[n];$ that is
 $\fP$ is {\it closed with respect to direct sums
 and unitary similarity.} \index{closed with respect to direct sums} 
 \index{closed with respect to unitary similarity} 

 Let $M(\C)$ denote the sequence $(M_n(\C))$
 and let $M(\C)^2$ denote the sequence $(M_n(\C)^2).$
 A \df{free analytic function} $f:\fP\to M(\C)$ is a sequence $(f[n])$
 of analytic functions $f[n]:\fP[n]\to M_n(\C)$ that
 respects direct sums and unitary similarities. That is,
 given $X\in \fP[n]$ and $Y\in \fP[m]$ and a unitary matrix $U\in M_n(\C),$
\[
 f[n+m](X\oplus Y)=f[n](X)\oplus f[m](Y)
\]
and
\[
f[n](U^*XU)= U^*f[n](X)U.
\]
 Typically we write $f$ in place of $f[n].$  The general
 definition of a free analytic function appears in 
 Subsection~\ref{s:freefun} below. 
 While it may not immediately appear so, free analytic
 functions are  the natural (freely) non-commutative analogs of
 analytic functions in several complex variables.

 A \df{free analytic mapping} $\varphi=(\varphi_1,\varphi_2):\fP\to\fP$ 
  is a pair of free analytic functions $\varphi_j:\fP\to M(\C)^2$
 such that 
\[
 \varphi(X)=(\varphi_1(X),\varphi_2(X)) \in \fP
\]
 for all $X\in \fP.$ An \df{automorphism} $\varphi$ of $\fP$
 is a free analytic mapping $\varphi:\fP\to\fP$ 
 for which there exists a free analytic mapping 
 $\psi:\fP\to\fP$ such that $\psi(\varphi(X))=X=\varphi(\psi(X))$
 for $X\in \fP.$

 Given $\gamma=(\gamma_1,\gamma_2)\in \CC^2$ with 
 $|\gamma_j|=1,$ the function $f(x) = (\gamma_1 x_1,\gamma_2 x_2)$
 is  automorphism of $\fP.$ Likewise, 
 $f(x)=(\gamma_2x_2,\gamma_1x_1)$ is an automorphism. 
 We call these automorphism \df{trivial automorphisms}. 
Theorem~\ref{t:main} below is the main result of this paper.

\begin{theorem}
 \label{t:main}
  If  
 \begin{enumerate}[(i)]
  \item $C_1$ and $C_2$ are invertible; and
  \item  the C-star algebras generated by
  $\{C_1^*C_1,C_2^*C_2\}$ and $\{C_1C_1^*, C_2C_2^*\}$
  are all of $M_s(\CC),$
 \end{enumerate}
  then the automorphisms of $\fP$ are trivial.
\end{theorem}

We wish to highlight two other contributions of this article. In Proposition~\ref{p:free-fun-alt} we show that the definition of free analytic function given here, which is tailored to the  study of  maps on free domains, coincides with other formulations in the literature and in particular such functions respect intertwinings. It is shown in \cite[Proposition~2.5]{proper}, assuming only continuity (and not analyticity) that a free function (as otherwise defined here) is in fact analytic. On the other hand, the analytic assumption is natural and the proof of Proposition~\ref{p:free-fun-alt} is rather simpler than that given  in \cite{proper}. The proof of Proposition~\ref{p:free-fun-alt} is modeled after arguments found in \cite{meric}. Proposition~\ref{p:freecirc} characterizing spectraballs is from \cite{circular}. Here we provide  an alternate proof.

 The remainder of this introduction contains more complete definitions of free spectrahedra and spectraballs, free analytic functions and maps, as well as background and motivation for studying the automorphism group of $\fP$. Preliminary results are contained in Section~\ref{s:prelims} and the proof of Theorem~\ref{t:main} appears in Section~\ref{s:proof}.

\subsection{Free polynomials and their evaluations}
 Fix a positive integer $\vg.$ Let $x=(x_1,\dots,x_\vg)$ denote $\vg$ freely non-commuting variables and let $\xx$ denote the semigroup of words in $x$ with $\varnothing,$ the enpty word, playing the role of the identity.   The length of the empty word is
 $0$ and otherwise the \df{length of a word}
\begin{equation}
 \label{d:word}
 w= x_{j_1} x_{j_2} \cdots x_{j_m}
\end{equation}
 is $m$ denoted $|w|=m.$

  For positive integers $n$, let $M_n(\C)^\vg$  denote the set 
of $\vg$-tuples $X=(X_1,\dots,X_\vg)$ of $n\times n$ matrices
 with entries from $\C$.  Let $M(\C)^{\vg}$ denote the sequence $(M_n(\C)^\vg).$  Given 
  a tuple $X=(X_1,\dots,X_\vg)\in M(\C)^\vg,$ let
\[
 X^w = X_{j_1} X_{j_2} \cdots X_{j_m},
\]
 with $w\in \xx$ as in equation~\eqref{d:word}.
 Thus $X^w$ is the \df{evaluation} of the word $w$ at the
 tuple $X.$ This evaluation extends to the \df{free algebra} of \df{free polynomials}  $\C\xx$  \index{$\C\xx$}
 equal the   $\C$ linear combinations of elements of $\xx.$ Elements  $p\in \C\xx$ have the form
\begin{equation}
 \label{e:polyp}
 p=\sum_{w\in\xx} p_w w,
\end{equation}
 where the sum is finite.  The polynomial $p$ evaluates at $X\in M(\CC)^\vg$ as
\[
 p(X) =\sum_{w\in \xx} p_w X^w.
\]
 A matrix-valued free polynomial can be viewed either as a matrix with polynomial entries or a polynomial with matrix coefficients. In the latter case, given positive integers $d,e$ and $p_w\in M_{d,e}(\CC),$ the finite sum in equation~\eqref{e:polyp} 
 is a matrix valued polynomial. To evaluate this $p$ at a tuple $X\in M(\CC)^{\vg}$  we will make use of the (Kronecker) tensor product $S\otimes T$ of matrices $S$ and $T,$ setting 
\[
 p(X)=\sum p_w \otimes X^w.
\]

\subsection{Free spectrahedra}
Given  $A\in M_{d\times e}(\C)^\vg$,  let $\Lambda_A$ denote 
the \df{homogeneous linear pencil}  \index{$\Lambda_A(x)$}
$\Lambda_A(x)=\sum_j A_jx_j.$ It  
evaluates at 
$X\in M_n(\C)^\vg$
as
\[
\Lambda_A(X) = \sum_{j=1}^\vg
   A_j \otimes X_j \in M_{d\times e}(\C)\otimes M_n(\C). 
\]
In the case  $A$ is square  ($d=e$), we let \index{$L_A$}
\[
\begin{split}
 L_A(X) & = I_d\otimes I_n 
 +\Lambda_A(X) +\Lambda_A(X)^*\\
 & = I+\sum A_j \otimes X_j
   + \sum A_j^* \otimes X_j^* \in M_d(\C)\otimes M_n(\C).
\end{split}
\]

 The  set $\cD_A[1]\subseteq \C^\vg$ consisting  of  $x\in\C^\vg$
  such that $L_A(x)\succ 0$ is a \df{spectrahedron}. 
 Spectrahedra are basic objects in a number of areas of mathematics;
 e.g.~semidefinite programming, convex optimization \cite{WSV}
  and real algebraic geometry \cite{BPT}.

The  \df{free spectrahedron} determined by
$A\in M_d(\C)^\vg$ is the sequence of sets $\cD_A= (\cD_A[n])$, where\index{$\cD_A$}
\[
\cD_A[n] =\{X\in M_n(\C)^\vg: L_A(X)\succ 0\},
\]
 and $T\succ 0$ indicates that the square matrix $T$
 is positive definite (hermitian with positive eigenvalues). 
 Observe  that $\fP$ is the free spectrahedron {$\cD_R,$} where \index{$\cD_R$}
\begin{equation}
\label{e:R}
 R_1 =\begin{pmatrix} 0&C_1&0&0\\0&0&0&0\\0&0&0&C_1\\0&0&0&0\end{pmatrix},
  \ \ \
R_2 = \begin{pmatrix} 0&0&C_2&0\\0&0&0&C_2\\0&0&0&0\\0&0&0&0\end{pmatrix}.
\end{equation}

 A free spectrahedron $\cD_A$ 
 is not determined by the spectrahedron $\cD_A[1].$ See Proposition~\ref{p:fPp-is-a-ball}.
Free spectrahedra are canonical objects in the theories
 of operator systems and spaces and completely
 positive maps. They are related to quantum channels from
 quantum information theory. That they  arise naturally 
 in certain systems engineering problems governed by a signal
 flow diagram \cite{convert-to-matin, emerge, SIG}
 also provides motivation for  studying free spectrahedra.

\subsubsection{Spectraballs}
  Given a tuple $G=(G_1,\dots,G_\vg)$ of $d\times e$ matrices, the sequence
 $\cB_G= (\cB_G[n])_n$  defined by \index{$\cB_G$}
\[
 \cB_G[n] =\{X\in M_n(\C)^\vg: \| \sum_{j=1}^\vg G_j\otimes X_j\|< 1\}
\]
is a \df{spectraball}. The spectraball at \df{level} one,   $\cB_G[1],$
 is a rotationally invariant convex subset of $\C^\vg.$ 
 The spectraball $\cB_G$ is a spectrahedron since 
 $\cB_G=\cD_B$ for $B=(\begin{smallmatrix}0&G\\0&0\end{smallmatrix})$.
 Under the hypotheses of Theorem~\ref{t:main},  the  spectrahedron $\fP$ is not a spectraball. See 
 Proposition~\ref{p:fPp-is-not-a-ball}.

 A spectrahedron $\cD_A$ has  its naturally \df{associated spectraball},
\begin{equation}
\label{e:specballA}
 \cB_A =\{X: \|\Lambda_B(X)\|<1\}=\{X: \begin{pmatrix} 0 & X\\0&0\end{pmatrix} \in \cD_A\}.
\end{equation}
 The spectraball $\cB_R$ associated to $\fP=\cD_R$ plays an important role in this article.

\subsubsection{Free sets}
 A \df{free set} $\sS\subseteq M(\C)^\vg$ is a sequence
 $\sS=(\sS[n])_n$  such that $\sS[n]\subseteq M_n(\C)^\vg$
 for each positive integer $n$ and such that $\sS$ is 
 \df{closed with respect to direct and unitary similarity}:
\begin{enumerate}[(i)]
\item   if
 $X\in \sS[n]$ and $Y\in \sS[m],$   then 
\[
X\oplus Y = \left ( X_1\oplus Y_1,\dots, X_\vg \oplus Y_{\vg}\right) \in \sS[n+m],
\]
where
\[
 X_j\oplus Y_j = \begin{pmatrix} X_j & 0 \\ 0 & Y_j \end{pmatrix};
\]
\item if $X\in \sS[n]$  and $U\in M_n(\C)$ is unitary, then
\[
U^* X U = \left ( U^* X_1U, \dots, U^* X_\vg U \right)\in \sS[n].
\]
\end{enumerate}
 We say $\sS$ is \df{open} if each $\sS[n]$ is open and $\sS$ is \df{bounded}
 if there exists a $\kappa$ such that for all $n$ and $X\in \sS[n],$
 the $n\times ng$ matrix
\[
 \begin{pmatrix} X_1 & \dots & X_\vg \end{pmatrix}
\]
 has norm at most $\kappa.$

 A free spectrahedron is an open  free set.

\subsection{Free analytic functions}
 \label{s:freefun}
 Given an open free set $\sS,$  a  \df{free function} 
 $f:\sS\to M(\C)$  is a sequence $f=(f[n]),$ where \index{$f[n]$}
 $f[n]:\sS\mapsto M_n(\C),$ that satisfies the axioms
\begin{enumerate}[(i)]
 \item  \label{i:sums}
 if $X\in \sS[n]$ and $Y\in \sS[m],$ then
\[
 f[n+m](X\oplus Y) = f[n](X)\oplus f[m](Y); 
\]
\item \label{i:usim}
  if $X\in \sS[n]$ and $U$ is an $n\times n$ unitary matrix, then
\[
 f[n](U^* XU)=  U^*f[n](X) U. 
\]
\end{enumerate}
 Thus free functions \df{respect direct sums and unitary similarity}.
   The free function $f$ is \df{analytic} if each $f[n]$ is analytic. 
  We typically write $f$ in place of $f[n].$

Turning to examples, free polynomials are evidently free analytic functions.  A free rational function $r$  (regular at $0$) is a free analytic function that has a realization formula; that is, there exists a positive integer $e,$  a tuple  $A\in M_{e}(\C)^\vg$ and vectors $c,b\in \C^e,$ such that 
\[
 r(x) = c^* (I-\Lambda_A(x))^{-1} b.
\]
 The natural domain of  $r$ consists of those tuples $X\in M_n(\C)^\vg$ for which
 $I-\Lambda_A(X)$ is invertible and for such an $X,$ 
\[
 r(X) = (I_d\otimes c)^* 
 \left (I_d\otimes I_n -\Lambda_A(X)\right )^{-1} \, (I_d\otimes b)
  \in M_n(\C).
\]
 See \cite{KVV} for further information about free functions.

There are two other formulations of free functions that are equivalent and more common in the literature.  They do not not assume analyticity, but rather have it as a consequence of mild additional assumptions such as continuity or boundedness, in which case they are equivalent to the formulation adopted here.  In one formulation, item~\ref{i:usim} is replaced by the hypothesis that $f$ \df{respects similarities}: if $X\in \sS[n]$ and $T$ is an invertible matrix such that $T^{-1}XT\in \sS[n]$, then $f(T^{-1}XT) = T^{-1} f(X) T$.  The other formulation replaces items~\ref{i:sums} and \ref{i:usim} with the single axiom that $f$ \df{respects intertwinings}: if $X\in \sS[n]$ and $Y\in \sS[m]$ and $\Gamma$ is an $m\times n$ matrix such that $\Gamma X=Y\Gamma$, then $\Gamma f(X)=f(Y)\Gamma$. 
 Proposition~\ref{p:free-fun-alt} is a variation on \cite[Lemma~3.5]{meric}.
 A proof appears in Subsection~\ref{s:nilp}. See also \cite[Proposition~2.5]{proper}.

\begin{proposition}
 \label{p:free-fun-alt}
 If $\sS$ is a free open set and $f:\sS\to M(\C)$  is a free analytic function, then $f$ respects intertwinings.
\end{proposition}

\subsection{Automorphisms of free spectrahedra}
 Given free sets $\sS\subseteq M(\C)^\vg$ and $\sT\subseteq M(\C)^\vg$
 a \df{free map}  $f:\sS\to \sT$ is an $\vg$-tuple
 $f=(f^1,\dots,f^\vg)$ of free functions $f:\sS\to M(\C)$ 
 such that $f(X)\in \sT$ for all $X\in \sS.$  A \df{bianalytic}
 map $f$ between free spectrahedra $\cD_A$ and $\cD_B$
 is a free map $f:\cD_A\to\cD_B$ for which
 there exists a free analytic mapping $g:\cD_B\to\cD_A$
 such that $g(f(X))=X$ and $f(g(Y))=Y$ for all $X \in \cD_A$
 and $Y\in\cD_B.$  We refer to  $g$ as  the \df{inverse}
 of $f$ and write $g=f^{-1}.$  A natural problem, from several
  different  perspectives, is to determine the bianalytic
 maps between two free spectrahedra. This 
 problem is the free analysis analog of rigidity phenomena
 in several complex variables and from this perspective
 it is expected  that two free spectrahedra are rarely
 bianalytic.

 The paper\cite{bestmaps} determines, under certain
 generic   irreducibility inspried  hypotheses on $A$ and $B,$
 the tuples $(A,B,f)$ of bianalytic maps $f:\cD_A\to \cD_B.$
 It turns out that $A$ and $B$ are closely linked and
 $f$ has a highly algebraic description. In that same
 paper, bianalytic maps between spectraballs are determined
 without any additional hypotheses. Again, these maps
 have a highly algebraic description. 

 An \df{automorphism} $f$ of a free spectrahedron $\cD_A$ is a bianalytic
 map $f:\cD_A\to\cD_A.$  
 If $f,g:\cD_A\to \cD_B$ are bianalytic, 
 then  the map $g^{-1}\circ f:\cD_A\to\cD_A$ is an automorphism. 
 Thus, the automorphism group of $\cD_A$ places
 constraints on the bianalytic maps $f:\cD_A\to\cD_B.$
 
 A free set $\sS$  is \df{circularly symmetric} if 
\[
 \gamma X=(\gamma X_1,\dots, \gamma X_\vg)\in \sS
\]
whenever $X\in \sS$ and $\gamma \in \CC$ is unimodular.
 A natural class of free spectrahedra not covered by
 the results  in \cite{bestmaps} 
 are those with circular
 symmetry that are not spectraballs. In particular,
 those $\fP$ satisfying the hypotheses of Theorem~\ref{t:main} provide
 examples of a spectrahedron whose automorphism
 group is not yet classified.  

 We are now in a position to provide an overview of the proof of Theorem~\ref{t:main}.
 A somewhat routine argument shows if $\varphi$ is a linear automorphism of $\fP,$ then $\varphi$ is trivial.
 A consequence of the free analog of the  Caratheodory-Cartan-Kaup-Wu
 from \cite{proper} is the following. If $\varphi:\fP\to\fP$ is an automorphism and $\varphi(0)=0,$ then $\varphi$ is linear.  The strategy employed here to show  if $\varphi$ is an automorphism of $\fP,$ then $\varphi(0)=0$ is novel, using algebraic aspects of the theory of spectraballs.
 Given an  automorphism $\varphi:\fP\to\fP$
 with  $\varphi(0)\in \C^2$  not necessarily $0,$
 we construct a  tuple $B$ described solely in terms of 
 $\varphi(0) \in \C^2$ and $\varphi^\prime(0)\in M_2(\C)$
 such that   $\cB_R=\cB_B.$  This equality is analyzed
 using algebraic results from \cite{bestmaps}, ultimately concluding
 $s=1.$ When  $s=1,$ the spectrahedron $\fP[1]$ is the generalized complex ellipsiod
 $\{z=(z_1,z_2)\in \C^2: |z_1|+|z_2|<1\}.$
 It is known \cite{Kodama,first-steps}, using very different 
 techniques than those here,  that the 
 analytic automorphisms of this complex ellipsiod are trivial. 
  In particular,  $\varphi(0)=0.$ 
 We give a self contained proof of this fact using results developed in this paper.

 By contrast, the proof strategy
 for classifying the automorphism group of
 hyper-Reinhardt domains, defined below and in \cite{MT,hyper}, 
 proceeds via the dynamics of composing automorphisms
 and using results related to the classical Caratheodory
 interpolation theorem.

\section{Preliminary results}
\label{s:prelims}
This section collects preliminary and ancillary results to Theorem~\ref{t:main}. Subsection~\ref{s:infP} provides alternate characterizations for membership in $\fP$ and establishes that, under the hypotheses of Theorem~\ref{t:main},  $\fP$ is neither a spectraball  nor a  hyper-Reinhardt domain. Additionally, an alternate proof of a results from \cite{circular} is given. Subsection~\ref{s:BR} gathers facts about the spectraball $\cB_R$ associated to $\fP=\cD_R.$ Ball minimality from \cite{bestmaps} is reviewed in Subsection~\ref{s:bmin}, where a ball minimal tuple $E$ such that $\cB_E=\cB_R$ is identified. Subsection~\ref{s:nilp} discusses the evaluation of a free function defined near $0$ on nilpotent tuples. The proof of Proposition~\ref{p:free-fun-alt} appears in Subsection~\ref{s:free-fun-alt}.

\subsection{Membership in $\fP$}
\label{s:infP}
 A matrix $T$ is a \df{strict contraction} if $\|T\|<1,$
 where $\|T\|$ is the operator norm of $T.$

\begin{proposition}
\label{p:fP-alt} For $X=(X_1,X_2)\in M(\C)^2$ and $Y_j=C_j\otimes X_j,$  the following are equivalent.
\begin{enumerate}[(a)]
 \item $X\in \fP;$
 \item the matrix
\[
 T(X):=\begin{pmatrix} Y_1^* & Y_2 \\ Y_2^* & Y_1 \end{pmatrix}
\]
is a strict contraction;
\item \label{i:fP-alt3}
 the matrix
\[
\begin{split}
 \cL^\prime(X)&  :=  \begin{pmatrix} I-Y_1^*Y_1 -Y_2Y_2^* &
      -Y_1^*Y_2-Y_2Y_1^* \\ -Y_2^*Y_1^*-Y_1Y_2^* &
      I-Y_1Y_1^*-Y_2^*Y_2\end{pmatrix} \\
   & = I  -T(X)T(X)^*
\end{split}
\]
 is positive definite;
\item \label{i:fP-alt4}
  The matrix
\[
\begin{split}
 \cL_*^\prime(X) & :=  \begin{pmatrix} I-Y_1Y_1^* -Y_2Y_2^* &
      -Y_1Y_2-Y_2Y_2\\ -Y_2^*Y_1^*-Y_1^*Y_2^* &
      I-Y_1^*Y_1-Y_2^*Y_2\end{pmatrix} \\
   & = I  -T(X)^*T(X)
\end{split}
\]
 is positive definite.
\end{enumerate}
\end{proposition}

\begin{proof}
 A  selfadjoint  block square matrix
\[
 M =\begin{pmatrix} A & B\\ B^* & D\end{pmatrix}
\]
 with $A$ positive definite is positive definite
 if and only if the Schur complement of its $(1,1)$ block,
\[
   S=D-B^*A^{-1}B 
\]
 is positive definite. 

 By definition, $X\in \fP$ means $\cL(X)\succ0.$
 Taking the Schur complement of $\cL(X)$ first of the $(1,1)$ block 
 entry then off the $(3,3)$ block entry (an identity matrix) shows
  $\cL(X)\succ 0$ if and only if $\cL^\prime(X)\succ 0$ using 
 the observation at the outset of this proof.  A direct computation 
 shows $\cL^\prime(X)=I-T(X)T(X)^*.$ 

 To complete the proof observe,
  $I-T(X)T(X)^*\succ 0$ if and only if $\|T(X)\|<1$ if and only 
 if $I-T(X)^*T(X)\succ 0$ if and only if $\cL^\prime_*(X)\succ 0.$
\end{proof}

 Specializing to the case of two variables, 
 a spectrahedron $\cD\subseteq M(\C)^2$ is  
 \df{hyper-Reinhardt} if there exists a tuple
 $G=(G_1,G_2)$ of matrices of compatible sizes 
 such that, with
\[
 A_1 =\begin{pmatrix} 0 & G_1 &0\\0&0&0\\0&0&0\end{pmatrix},  \ \
 A_2 =\begin{pmatrix} 0&0&0\\0&0&G_2\\0&0&0\end{pmatrix},
\]
we have $\cD=\cD_A.$ A  hyper-Reinhardt 
 free spectrahedra is  circular, but not necessarily 
 a spectraball.  The automorphisms of hyper-Reinhardt
 free spectrahedra  are determined in \cite{MT, hyper}.

\begin{proposition}
  \label{p:fPp-is-not-a-ball}
 Under the hypotheses of Theorem~\ref{t:main},  the spectrahedron $\fP$ is neither  hyper-Reinhardt nor a spectraball.
\end{proposition} 

 The proof of Proposition~\ref{p:fPp-is-not-a-ball} uses
 one direction of the following result from \cite{circular}.
 
\begin{proposition}
 \label{p:freecirc}
 A spectrahedron $\cD_A$ is a spectraball if and only if
 for each positive integer $n,$ each $X\in \cD_A[n]$
 and each unitary matrix $U\in M_n(\C),$ the
 tuple $UX\in \cD_A.$
\end{proposition}

\begin{proof}
  If $\cD_A=\cB_B$ is a spectraball, then it is immediate
 that $X\in \cD_A[n]$ and $U$ unitary implies
 $UX\in \cD_A.$  

 Conversely suppose for each positive integer $n,$ each $X\in \cD_A[n]$
 and each unitary matrix $U\in M_n(\C),$ the
 tuple $UX\in \cD_A[n].$ To prove $\cD_A$ is spectraball, it suffices to show $\cD_A=\cB_A.$ To this end, 
  let $X\in M_n(\C)^\vg$ be given.  Let
\[
 U=\begin{pmatrix} 0&I_n\\I_n &0\end{pmatrix}.
\]
 and observe
 $X\in \cD_A$ if and only if   $0_n\oplus X\in \cD_A[2n]$
 if and only if 
\[
 U (0_n\oplus X)= U\begin{pmatrix} 0&0\\0& X\end{pmatrix}
 =  \begin{pmatrix} 0&X\\0&0\end{pmatrix}\in \cD_A[2n]
\]
if and only if $X\in \cB_A[n]$ (see  Equation~\ref{e:specballA}). 
 Thus $\cD_A=\cB_A$ and the proof is complete.
\end{proof}

The following lemma is also needed for the proof of Proposition~\ref{p:fPp-is-not-a-ball}.

\begin{lemma}
 \label{l:CstarC}
 Under the hypotheses of Theorem~\ref{t:main}, $C_1^*C_1+C_2^*C_2\succ I_s$ and $C_1 C_1^* + C_2 C_2^*\succ I_s.$
\end{lemma}

\begin{proof}
 If $C_1^*C_1 +C_2^*C_2\preceq I_s,$  then  the kernel $K$  of $I-C_1^*C_1$ is non-trivial and orthogonal to the range of $C_2^*C_2.$ Hence $K$ reduces both $C_1^*C_1$ and $C_2^*C_2$ and hence the C-star algebra they generate. Thus $C_1^*C_1+C_2^*C_2\succ I_s.$ Likewise $C_1C_1^*+C_2C_2^*\succ I_s.$
\end{proof}

\begin{proof}[{Proof of Proposition~\ref{p:fPp-is-not-a-ball}}]
 A routine computation shows, if $\cD\subseteq M(\C)^\vg$ is a hyper-Reinhardt 
 free spectrahedron, 
  $X\in \cD[n]$ and  $W_0,W_1,W_2\in M_n(\C)$ are unitary matrices, then
\[
 W\cdotb X := (W_0^*XW_1, W_1^*XW_2) \in \cD[n].
\]

 Another routine computation and an appeal to Proposition~\ref{p:fP-alt} shows 
\begin{equation} \label{e:Xnotball}
 X=(X_1,X_2) =
 \left ( \begin{pmatrix} 1&0\\0&0\end{pmatrix}, 
  \begin{pmatrix} 0&0\\0&1\end{pmatrix} \right ) 
\end{equation}
 is in the boundary of $\fP.$  Let
\begin{equation} \label{e:Unotball}
 U=\begin{pmatrix} 0&1\\1&0\end{pmatrix}
\end{equation}
 and $(W_0,W_1,W_2)=(I,U, U^2).$ Thus each $W_j$ is unitary and
\[
 W\cdotb X = \left( \begin{pmatrix} 0&1\\0&0\end{pmatrix},
  \begin{pmatrix}0&1\\0&0\end{pmatrix}\right ).
\]
  By item~\ref{i:fP-alt4} or Proposition~\ref{p:fP-alt} 
 and  Lemma~\ref{l:CstarC}, $W\cdotb X$  is not in the 
 closure of $\fP.$   Hence $\fP$ is not hyper-Reinhardt.

 If $\fP$ is a spectraball,
  $X\in \fP[n]$ and $U$ is an $n\times n$ unitary matrix,
 then $UX=(UX_1,UX_2)$ is also in $\fP[n]$
 by Proposition~\ref{p:freecirc}. Choose
 $X_1,X_2$ and $U$ as in equations~\eqref{e:Xnotball}
 and \eqref{e:Unotball} and note, with 
 $\cL^\prime$ as in Proposition~\ref{p:fP-alt},
\[
 \cL^\prime(UX) = I - \begin{pmatrix}  
   C_1C_1^*+C_2^*C_2 &0&0&0\\0&0&0&0\\
   0&0&0&0\\0&0&0& C_1^*C_1 + C_2C_2^* \end{pmatrix},
\]
 since $X_1^*X_2=0=X_2X_1^*.$ If $UX$ is in the boundary of $\fP,$
 then $I-C_1^*C_1-C_2C_2^*\succeq 0.$ By assumption, there is a vector
 $x$ such that $C_1^*C_1x=x,$  from which it follows that $C_2^*x=0,$
 contradicting the assumption that $C_2$ is invertible.
\end{proof}

\subsection{The spectraball $\hfP$}
\label{s:BR}
 Let $\{e_1,e_2\}$ denote the standard basis for $\CC^2$ and let $e_j^T$ denote the transpose of $e_j.$
 Let $E^r\in M_2(\C)^2$ denote the tuple, \index{$E^r$} \index{$E^c$}
\begin{equation*}
 %\label{d:Er}
  E^r_1 = e_1^T \otimes C_1= \begin{pmatrix} C_1&0\end{pmatrix}, 
 \ \ \  E^r_2 =e_2^T \otimes C_2  =\begin{pmatrix} 0 & C_2 \end{pmatrix}
\end{equation*}
 and let $E^c\in M_2(\C)$ denote the tuple
 $E^c=(E^c_1,E^c_2)$ where  $E_j^c=e_j\otimes C_j.$

 The spectraball associated to $\fP$
 is $\cB_R,$ where $R$ is defined in equation~\eqref{e:R}.
 Since the first row and last column of $R_j$ are zero,
 $\cB_R=\cB_E,$ where 
\begin{equation}
\label{d:E}
 E=(E_1,E_2)=\left ( \begin{pmatrix} C_1&0&0\\0&0&0\\0&0&C_1\end{pmatrix},
   \begin{pmatrix} 0&C_2&0\\0&0&C_2\\0&0&0\end{pmatrix}
    \right ) = \begin{pmatrix} E^r &0\\0 & E^c \end{pmatrix}.
\end{equation}

\begin{proposition}
 \label{p:fPp-is-a-ball}
   The equality $\hfP=\cB_E= \cB_{E^r}\cap \cB_{E^c}$ holds.
 In particular, $X\in \hfP$ if and only if both
\[
 \begin{pmatrix} C_1\otimes X_1 & C_2\otimes X_2 \end{pmatrix}, \ \ \ 
 \begin{pmatrix} (C_1\otimes X_1)^*&(C_2\otimes X_2)^* \end{pmatrix}
\]
 are strict contractions.
 On the other hand,  under the hypotheses of Theorem~\ref{t:main}, $\cB_{E^j}[2]\not \subseteq \cB_{E^k}[2]$
 for $j,k\in \{r,c\}$ and $j\ne k.$
\end{proposition}

\begin{proof}
 The equalities $\hfP=\cB_E= \cB_{E^r}\cap \cB_{E^c}$ are  immediate from equation~\eqref{d:E}.
 
 Using Lemma~\ref{l:CstarC}, 
 choose $\rho<1$ such that $\rho^2(C_1^*C_1+C_2^*C_2)\not \preceq I.$
 The tuple 
\[
 X =\rho  \left ( 
   \begin{pmatrix} 0&1\\0&0\end{pmatrix}, \begin{pmatrix}
     0&0\\0&1\end{pmatrix} \right )
\]
 is in  $\cB_{E^r}$ but not $\cB_{E^c}.$ 
 Hence $\cB_{E^r}[2]\not\subseteq \cB_{E^c}[2].$
 A similar argument with
\[
 X =\rho  \left ( 
   \begin{pmatrix} 1&0\\0&0\end{pmatrix}, \begin{pmatrix}
     0&1\\0&0\end{pmatrix} \right )
\]
 shows $\cB_{E^c}[2]\not\subseteq \cB_{E^r}[2].$
\end{proof}

\begin{remark}\rm
 In the terminology of \cite{bestmaps}, the direct sum
 $E^r\oplus E^c$ is {\it irredundant}. 
 \qed
\end{remark}

\subsection{Ball minimality}
\label{s:bmin}
 The discussion of ball-minimality here is borrowed
 from \cite{bestmaps}. In particular, Lemma~\ref{l:bmin} below is
 excerpted from \cite[Lemma~3.1]{bestmaps}.
A  $\vg$-tuple $E=(\gG_1,\dots,\gG_\vg)$ of $d\times e$ matrices is
\df{ball-mimimal} for $\cB_\gG$ if $F\in M_{k\times \ell}$ and
 $\cB_F=\cB_\gG$ implies $d\le k$ and $e\le \ell.$
 As examples,  it is
 immediate that both $E^r$ and $E^c$ are ball-minimal
 (for $\cB_{E^r}$ and $\cB_{E^c}$ respectively). 
 Tuples $\gG,F\in M_{d,e}(\C)^\vg$ are \df{ball-equivalent}
 if there exists unitary matrices $U\in M_{e,e}(\C)$
 and $V\in M_{d,d}(\C)$ such that $F=V^* \gG U.$
 Observe that, in this case, $\cB_F=\cB_\gG.$
 A tuple $A\in M_d(\C)^\vg$ is minimal for $\cD_A$ if $B\in M_k(\C)^\vg$ and $\cD_B=cD_A$ implies $d\le k.$

\begin{lemma}
 \label{l:bmin}
   Ball minimal tuples exists; that is, if $\cB\subseteq M(\C)^\vg$ is a spectraball, then there exists $d,e$ and  a ball minimal tuple  
  $\gG\in M_{d,e}(\C)^\vg$ such that $\cB=\cB_\gG.$

 Let $\gG\in M_{d,e}(\C)^\vg$ be given. 
\begin{enumerate}[(i)]
\item  If  $\gG$ ball minimal, $F\in M_{k\times \ell}(\C)^\vg$ 
and $\cB_\gG=\cB_F$, then there
 is a tuple $J\in M_{(k-d)\times (\ell-e)}(\C)^\vg$
   and unitaries $U,V$  of sizes $k\times k$ and $\ell\times \ell$
 respectively such that
 $\cB_\gG\subseteq \cB_J$ and 
\begin{equation*}
 %\label{e:FUERV}
 F = U\begin{pmatrix} \gG & 0 \\ 0 & J \end{pmatrix} V.
\end{equation*}
In particular,
 \begin{enumerate}[\rm (a)]
  \item $d\le k$ and $e\le \ell$;
 \item  if $F\in M_{d\times e}(\C)^\vg$ is ball minimal too, then $\gG$ and $F$ are ball-equivalent.
  \end{enumerate}
\item   $\gG\in M_{d,e}(\C)^\vg$ is ball minimal if and only if 
\[
  H=\begin{pmatrix} 0 & \gG\\0&0\end{pmatrix} \in M_{d+e}(\C)^\vg
\]
 is minimal for $\cD_H.$
\end{enumerate}
\end{lemma}

\begin{proposition}
 \label{p:E-is-ball-min}
 The tuple $E$ from equation~\eqref{d:E} is ball-minimal.
\end{proposition}

\begin{proof}
 Let 
\[
 H=\begin{pmatrix} 0 & E \\ 0 & 0 \end{pmatrix} \in M_{6s,6s}.
\]
 Thus $\cB_E=\cD_H$ and, by Lemma~\ref{l:bmin}, $E$ is ball minimal for $\cB_E$ if and only if $H$ is minimal for $\cD_H.$

 Let $\{\EE_{j,k}: 1\le j\le 6\}$ denote the matrix units
 for $M_6(\CC).$  Let
\[
 F_1 = [\EE_{1,2} + \EE_{5,6}]\otimes C_1, \ \ \
 F_2 =[\EE_{1,3}+\EE_{4,6}]\otimes C_2.
\]
Since $F$ and $H$ are unitarily equivalent, we have $\cD_H=\cD_F$ and moreover $H$ is minimal defining for $\cD_H$ if and only if $F$ is minimal for $\cD_F$ if and only if $E$ is ball minimal for $\cB_E.$

 Let $\cF$ denote the C-star algebra generated by
 $\{F_1,F_2\}.$ 
 Since $F_1F_1^* = [\EE_{1,1}+ \EE_{5,5}]\otimes C_1C_1^*$
 and $F_2F_2^*=[\EE_{1,1} +\EE_{4,4}]\otimes C_2C_2^*$
 and since, by assumption, $\{C_1C_1^*,C_2C_2^*\}$
 generates $M_s(\CC)$ as a C-star algebra, 
 for each $X\in M_s(\CC)$ there exist $Y_1,Y_2\in M_s(\CC)$
 such that $\cF$ contains 
 $\EE_{1,1}\otimes X + \EE_{4,4}\otimes Y_2   
  +\EE_{5,5} \otimes  Y_1.$
 Multiplying by $F_1F_1^*\,F_2F_2^*$ on the right, it follows that
 $\cF$ contains $\EE_{1,1}\otimes X C_1C_1^*C_2C_2^*.$
 Since $C_1C_1^* C_2C_2^*$ is invertible, it follows that $\cF$
 contains $\EE_{1,1} \otimes M_s(\CC).$  By considering
 $F_j^* \, [\EE_{1,1} \otimes M_s(\CC)] \, F_k,$
 for $j,k=0,1,2$ and $F_0=I,$ and arguing as above,
 it follows that $\cF$ contains 
 $\sum_{j,k=1}^3 \EE_{j,k}\otimes M_s(\CC);$ that
 is $\cF$ contains $\cF_1=M_{3s}(\CC)\oplus 0$ as a
 C-star subalgebra. By considering instead $F_j^*F_j$
 and using $\{C_1^*C_1,C_2^*C_2\}$ generates
 $M_s(\CC)$ as a C-star algebra,
 it follows that $\cF$ contains  $\cF_2=0\oplus M_{3s}(\CC)$
 as a C-star subalgebra. Hence $\cF=M_{3s}(\CC)\oplus M_{3s}(\CC).$
 In particular, aside from the trivial ones,
  the only reducing subspaces for $\cF,$
 equivalently, $\{F_1,F_2\},$ 
 are $\CC^{3s}\oplus\{0\}$ and $\{0\}\oplus \CC^{3s}.$

 Suppose $A\in M_k(\CC)^2$
 is \df{minimal defining} for $\cD_F.$ Thus $k\le 6s.$
 By \cite[Proposition~2.2]{circular}, there exist
 a tuple $J\in M(\CC)^{\vg}$ such that, up to unitary
 equivalence, $F=A\oplus J.$  Suppose $A\ne F.$
  It follows that either
 $A=(\EE_{1,2}\otimes C_1, \EE_{1,3} \otimes C_2)$
 or $A=(\EE_{4,6}\otimes C_1, \EE_{4,6}\otimes C_2).$
  But then, $\cD_A$ is either $\cB_{E^r}$ or
 $\cB_{E^c},$ contradicting the conclusion
 of Proposition~\ref{p:E-is-ball-min} 
 ($\{\cB_{E^r},\cB_{E^c}\}$ is irredundant).  Thus $F$ is minimal for $\cD_F$ and $E$ is ball minimal.
\end{proof}

\subsection{Nilpotent evaluations}
\label{s:nilp}
  Given $\delta>0,$ let
\[
B(0,\delta) =\{X\in M(\C)^\vg: \sum_{j=1}^\vg X_jX_j^*  <\delta^2\}
\]
 It is straightforward to verify $B(0,\delta)$ is an open
 free set.

 A (formal) \df{power series} $F$  in $\xx$ is an expression of the form
\begin{equation}
 \label{d:pow}
 F(x) = \sum_{w\in \xx} F_w w.
\end{equation}

 A tuple $X\in M(\C)^\vg$ is \df{nilpotent of order at most $m$}
 if $X^w=0$ for all words $w$  of length $m.$ 
 Theorem~\ref{t:pow-series} below
  can be found in \cite{KVV,KS,billvolume} among
 other places.

\begin{theorem}
 \label{t:pow-series}
 Suppose $S$ is an open free set and there is a $\delta>0$
 such that $B(0,\delta) \subseteq S.$ If $f:S\to M(\C)$
 is a free analytic function, then  there exists a formal power 
 series as in equation~\eqref{d:pow} such that if $X\in B(0,\delta),$
 then 
\[
 f(X)=F(X) =\sum_{\ell=0}^\infty \sum_{|w|=\ell} F_w X^w,
\]
 with the series converging in norm.  Further, $f$ extends
 analytically to all nilpotent tuples. In particular,
 if $X$ is nilpotent of order at most $m,$ then
\[
 f(zX) = \sum_{\ell=0}^m (\sum_{|w|=\ell} F_w X^w) z^\ell
\]
 for $z\in \CC.$
\end{theorem}

 Let \index{$S$}
\begin{equation*}
  %\label{d:S}
 S=\begin{pmatrix} 0 & 1 \\ 0 & 0 \end{pmatrix}.
\end{equation*}
 Thus, given $X=(X_1,X_2)\in M_n(\C)^2,$ the tuple 
\[
 S\otimes X =(S\otimes X_1, \, S\otimes X_2)
 = \left ( \begin{pmatrix} 0 & X_1 \\0 & 0 \end{pmatrix}, \,
 \begin{pmatrix} 0 & X_2 \\ 0 & 0  \end{pmatrix} \right)
 \in  M_{2n}(\C)^2 
\]
is nilpotent of order two. 
The evaluation of free map $f:\fP\to M(C)$ on 
 $S\otimes X$ takes a particularly simple form. 

\begin{proposition}
 If $f:\fP\to M(\C)$ is a free analytic function, then 
 $f$ extends uniquely to a function, still denoted $f,$
 defined on all tuples of the form $S\otimes X$
 for $X\in M(\C)^2.$ Moreover, 
  there exists $\ell_1,\ell_2\in \C$ such that
\[
 f(S\otimes X) = f(0) + S\otimes (\sum_{j=1}^2 \ell_j X_j) \in 
   M_2(\C)\otimes M_n(\C) =M_{2n}(\C).
\]
\end{proposition}

 We close this section with a proof of Proposition~\ref{p:free-fun-alt}. The ideas
 borrow freely from \cite{meric}.

\subsection{Proof of Proposition~\ref{p:free-fun-alt}}
\label{s:free-fun-alt}
%\begin{proof}[Proof of Proposition~\ref{p:free-fun-alt}]
 We are to show, if $\sS$ is a free open set and
 $f:\sS\to M(\C)$ is analytic and a free function in the sense that
 $f$ respects both direct sums and unitary similarity, then
 $f$ respects intertwinings. To this end, 
 suppose $X\in \sS[m]$ and $Y\in \sS[n]$ and $\Gamma$ is an
 $m\times n$ matrix such that $X\Gamma =\Gamma Y.$
 To show $f(X)\Gamma=\Gamma f(Y),$ 
 first observe we may replace $\Gamma$ with $t\Gamma$
 for any non-zero $t\in \C.$ Since $\sS$ is closed with respect
 to direct sums $Z=X\oplus Y\in \sS[n+m].$ Since $\sS[n+m]$ is open
 there is an $\epsilon>0$ such that if $T$ is an
 $(n+m)\times (n+m)$ matrix and $\|T-I\|<\epsilon,$ then
 $T$ is invertible and 
\[
 T^{-1} \, [X\oplus Y] \, T \in \sS[n+m].
\]

 For a matrix $M$ and $\rho\in \C$ with $|\rho|\, \|M\|<1,$ let
\[
 D_\rho(M) = (I-\rho^2 M^*M)^{\frac12}
\]
 and let, for $0\ne z\in \C,$
\[
 T_\rho(M)[z] = \begin{pmatrix}  D_\rho(M^*)  & \frac{\rho}{z} M\\
    -\rho \, z M^* & D_\rho(M) \end{pmatrix}.
\]
 For $|z|=1$ the matrix $T_\rho[M](z)$ is a version of the
 Julia matrix and is unitary.

 By choosing $\|\Gamma\|$ sufficiently small, we may assume,
 for all $0<\rho<1,$  and $\rho < |z| < \frac{1}{\rho},$
 that 
\[
  \|I- T_\rho(\Gamma)[z]\| <\epsilon.
\]
 Hence, 
\[
  T_\rho(\Gamma)[z]^{-1} \, [X\oplus Y] \, T_\rho(\Gamma)[z]
   \in \sS[n+m]. 
\]
 The function (for fixed $\rho$)
\[
 F(z) = f\left( T_\rho(\Gamma)[z]^{-1} \, [X\oplus Y] \,
  T_\rho(\Gamma)[z] \right )
 - T_\rho(\Gamma)[z]^{-1} \, f(X\oplus Y)  \, T_\rho(\Gamma)[z]
\]
 is defined and analytic on the annulus $\rho \le |z| \le \frac{1}{\rho}$
 and  vanishes for $|z|=1$ since $f$ respects unitary
 similarities. Thus $F$ vanishes identically. 
 Choosing
 $z=\rho$ and then letting $\rho$ tend to $0$ gives,
\[
 f(\begin{pmatrix} I & -\Gamma \\ 0& I \end{pmatrix} 
 \, \begin{pmatrix} X&0\\0& Y \end{pmatrix} \,
  \begin{pmatrix} I & \Gamma \\ 0& I \end{pmatrix} )\\
= 
 \begin{pmatrix} I & -\Gamma \\ 0& I \end{pmatrix} 
 \, f(\begin{pmatrix} X&0\\0& Y \end{pmatrix}) \,
  \begin{pmatrix} I & \Gamma \\ 0& I \end{pmatrix}.
\]
 Thus, using $f$ respects direct sums, 
\[
\begin{split}
 \begin{pmatrix} f(X) &0\\0 & f(Y) \end{pmatrix}
 &= f(\begin{pmatrix} X&0\\0& Y \end{pmatrix})
 = f(\begin{pmatrix} X&X\Gamma-\Gamma Y\\0& Y \end{pmatrix})
 \\ & = f(\begin{pmatrix} I & -\Gamma \\ 0& I \end{pmatrix} 
 \, \begin{pmatrix} X&0\\0& Y \end{pmatrix} \,
  \begin{pmatrix} I & \Gamma \\ 0& I \end{pmatrix} )
= 
 \begin{pmatrix} I & -\Gamma \\ 0& I \end{pmatrix} 
 \, f(\begin{pmatrix} X&0\\0& Y \end{pmatrix}) \,
  \begin{pmatrix} I & \Gamma \\ 0& I \end{pmatrix}
\\ &= 
 \begin{pmatrix} I & -\Gamma \\ 0& I \end{pmatrix} 
 \, \begin{pmatrix} f(X) &0\\0& f(Y) \end{pmatrix} \,
  \begin{pmatrix} I & \Gamma \\ 0& I \end{pmatrix}
 = \begin{pmatrix} f(X) & f(X)\Gamma -\Gamma f(Y)\\
    0&f(Y) \end{pmatrix}
\end{split}
\] 
 and the desired conclusion follows.
%\end{proof}

\section{A proof of Theorem~\ref{t:main}}
\label{s:proof} 
In this section, we present a proof of Theorem~\ref{t:main}. In Subsection~\ref{s:affine}, we construct a spectraballs   $\cB_B$ that is canonically associated with an automorphism $\varphi: \fP \to \fP$. In Subsection~\ref{s:tmd}, we collect consequences of the equality $\cB_B=\cB_R=\cB_E$ that are then used in Subsection~\ref{s:dichotomy} that $b = 0$. The proof concludes in Subsection~\ref{s:concludes}.

\subsection{An affine change of variables}
\label{s:affine}
 In this section a  spectraball $\cB_B$  canonically associated to an automorphism  $\varphi:\fP\to\fP$ is constructed.

 Suppose $\varphi:\fP\to\fP$ is bianalytic. The 
 mapping $\varphi$ has coordinate functions $\varphi_j$
 so that
\[
  \varphi(X) =(\varphi_1(X),\varphi_2(X))
\]
 for $X\in \fP.$

Express the  series expansions for $\varphi_j,$ 
  up to the first degree terms as
\begin{equation*}
 %\label{e:varphi}
\varphi_j(x)= b_j + \sum_k\ell_{j,k} x_k + \cdots.
\end{equation*}
 Since $\varphi$ is  bianalytic, 
 the matrix \index{$\fL$}
\begin{equation}
\label{d:fL}
 \fL=\begin{pmatrix} \ell_{1,1} & \ell_{1,2} 
       \\ \ell_{2,1} & \ell_{2,2}
   \end{pmatrix}
\end{equation}
 is invertible.

 Given a tuple $X\in M_n(\C)^2,$
\begin{equation*}
 %\label{e:varphiXnilp2}
 \varphi_j(S\otimes X) = \begin{pmatrix} b_j & \sum_k \ell_{j,k} X_k\\
 0 & b_j \end{pmatrix}.
\end{equation*}

 Let $\tb_j=b_jC_j.$
 Let  \index{$B_0$}
\begin{equation*}
%\label{d:B0}
%\begin{split}
 B_0  =\begin{pmatrix} I&\tb_1&\tb_2&0\\ \tb_1^*
    &I & 0 & \tb_2\\ \tb_2^* & 0 & I  & \tb_1\\
 0 & \tb_2^* & \tb_1^* & I\end{pmatrix}.
\end{equation*}
 Since $\varphi(0)=b\in \fP,$  the matrix $B_0=\cL(\varphi(0))$
  is positive definite
 and hence invertible.

 Let 
\begin{equation}
\label{d:Y}
 Y_j  = \begin{pmatrix} 0 &   \ell_{1,j}C_1  & \ell_{2,j}C_2  & 0\\
 0&0&0 & \ell_{2,j}C_2  \\
 0&0&0 & \ell_{1,j}C_1   \\ 0&0&0&0\end{pmatrix}
 = \sum_k \ell_{k,j} R_k,
%\end{split}
\end{equation}
 where $R$ is defined in equation~\eqref{e:R}.
Let $B=(B_1,B_2)$ denote the tuple defined by
\begin{equation*}
 %\label{d:Bj}
   B_j = B_0^{-\frac12} Y_j B_0^{-\frac12}.
\end{equation*}
Let $\Lambda_B(x)= B_1x_1 + B_2 x_2$ and let $\cB_B$ denote the
 resulting  spectraball,
\[
 \cB_B =\{X: \|\Lambda_B(X)\|<1\}.
\]

\begin{proposition} 
 \label{p:althfP}
 The equality $\hfP = \cB_B$ holds.
\end{proposition}

 Before proving Proposition~\ref{p:althfP} we record
 the following lemma.

 \begin{lemma}
 \label{l:althfP}
  There is a permutation matrix $\Sigma$ on 
 $\{1,2,\dots,8\}$ such that
\begin{equation}
 \label{e:althfP}
 (\Sigma^*\otimes I_s)  \cL(\varphi(S\otimes x)) (\Sigma\otimes I_s)
   = \begin{pmatrix}  B_0 & \Lambda_Y(x) \\ \Lambda_Y(x)^* & B_0
 \end{pmatrix}.
\end{equation}
\end{lemma}

\begin{proof}
 Let $\ty_j=(\sum_{k=1}^2 \ell_{j,k}x_k)\, C_j .$ 
 Identifying  $0$ with the 
 $0$ matrix of size $s\times s,$
\[
 \cL(\varphi(S\otimes x)) 
  = \begin{pmatrix} I&0&\tb_1&\ty_1&\tb_2&\ty_2&0&0\\
       0&I&0&\tb_1&0&\tb_2&0&0 \\
       \tb_1^* &0&I&0&0&0&\tb_2&\ty_2\\
      \ty_1^*&\tb_1^*&0&I&0&0&0&\tb_2\\
      \tb_2^*&0&0&0&I&0&\tb_1&\ty_1\\
      \ty_2^*&\tb_2^*&0&0&0&I&0&\tb_1\\
      0&0&\tb_2^*&0&\tb_1^*&0&I&0\\
     0&0&\ty_2^*&\tb_2^*&\ty_1^*&\tb_1^*&0&I  \end{pmatrix}.
\]
 Thus equation~\eqref{e:althfP} holds with $\Sigma$
 equal the tensor product
 of  the matrix associated to  the permutation
 $(1,3,5,7,2,4,6,8)$ with $I_s,$
 the $s\times s$ identity,  after noting that
\[
 \begin{pmatrix} 0&\ty_1&\ty_2&0\\0&0&0&\ty_2\\0&0&0&\ty_1\\
   0&0&0&0 \end{pmatrix} =  \Lambda_Y(x). \qedhere
\]
\end{proof}

\begin{proof}[Proof of Proposition~\ref{p:althfP}]
 Using the assumption that  $\varphi$ is an automorphism of $\fP,$
 Equation~\eqref{e:specballA}  and Lemma~\ref{l:althfP},
  a tuple $X$ is in $\hfP[n]$ if and only if $S\otimes X\in \fP[2n]$
 if and only if $\varphi(S\otimes X)\in \fP$ if and only
  if $\cL(\varphi(S\otimes X))\succ 0$ if and only if
\[
  \begin{pmatrix} B_0\otimes I_n & \Lambda_Y(X) \\ 
    \Lambda_Y(X)^*&B_0\otimes I_n \end{pmatrix} \succ 0
\]
 if and only if
\[
{\scriptsize
 0\prec
  \begin{pmatrix} B_0^{-\frac12}\otimes I_n &0 \\0& B_0^{-\frac12}\otimes I_n\end{pmatrix}
 \, \begin{pmatrix} B_0\otimes I_n & \Lambda_Y(X) \\ 
    \Lambda_Y(X)^*&B_0\otimes I_n \end{pmatrix} \,
 \begin{pmatrix} B_0^{-\frac12}\otimes I_n &0 \\0& B_0^{-\frac12}\otimes I_n\end{pmatrix}
 = \begin{pmatrix} I & \Lambda_B(X)\\ \Lambda_B(X)^*&I \end{pmatrix}
}
\]
 if and only if $X\in \cB_B[n].$
\end{proof}

\subsection{Analyzing the equality $\hfP=\cB_B$}
\label{s:tmd}
 In this section we deduce consequences of the equality the two representation $\cB_E$ and $\cB_B$ of $\cB_R$  appearing in Propositions~\ref{p:fPp-is-a-ball}  and \ref{p:althfP} respectively
 using the theory of spectraballs.

 Since, by Proposition~\ref{p:E-is-ball-min},
  $E$ is ball-minimal and $\cB_E=\cB_B,$ Lemma~\ref{l:bmin}
 implies there exist unitary matrices $U,V$ such that
\[
 U \begin{pmatrix} E&0\\0&J\end{pmatrix}V^* =B,
\]
 where $J=(J_1,J_2)\in M_s(\CC)^2$
 and $\cB_E \subseteq \cB_J.$ 
 Letting $\{e_1,\dots,e_4\}$
 denote the standard basis for $\CC^4$
 and  $H_j=\{e_j\otimes h: h\in \CC^s\},$
\[
\ker(B):= \cap_{j=1}^2 \ker(B_j) =B_0^{\frac12} H_1,
\]
 since $\cap_{j=1}^2 \ker(Y_j)=H_1.$  Thus
\[
 \begin{pmatrix} E_j &0\\0&J_j\end{pmatrix}V^* B_0^{\frac12} H_1 =0.
\]
 Since $E$ has no kernel, it follows that $J=0$ and
 $V^*B_0^{\frac12} H_1 = H_4.$   In particular,
\[
  U^* B_j V =\begin{pmatrix} E_j&0 \\ 0 & 0 \end{pmatrix}
\]
where the lower left $0$ matrix has size $s\times s.$

 Summarizing, there exist unitary matrices $U,V$ 
 such that
\[
 B_j = U \begin{pmatrix} E_j & 0 \\0&0\end{pmatrix} V^*.
\]
 Set $P=B_0^{\frac12}U$ and $Q^*=V^* B_0^{\frac12}.$  
 Thus $PP^*=B_0=QQ^*$ and
\[
 Y_j= P \begin{pmatrix} E_j&0\\0&0\end{pmatrix} Q^*.
\]

 Express $P$ and $Q$ in terms of their block columns as
\begin{equation*}
%\label{e:PnQ-1}
\begin{split}
 P & =\begin{pmatrix} p_1 & p_2 & p_3 & p_4\end{pmatrix}
     = \begin{pmatrix} P_{j,k} \end{pmatrix}_{j,k=1}^4 \\
 Q&=\begin{pmatrix} q_1 & q_2 & q_3 & q_4\end{pmatrix}
  =\begin{pmatrix} Q_{j,k}\end{pmatrix}_{j,k=1}^4,
\end{split},
\end{equation*}
 where $P_{j,k}$ and $Q_{j,k}$ are $s\times s$ matrices.
 With these notations, 
\begin{equation}
 \label{e:YvPQ-mat}
\begin{split}
  Y_1 = & P\, \begin{pmatrix} E_1&0\\0&0 \end{pmatrix} \,  Q^* 
   = P[(e_1e_1^* +e_3e_3^*)\otimes C_1]Q^*
    = p_1 C_1 q_1^* + p_3 C_1 q_3^* \\
 Y_2 = & P\, \begin{pmatrix} E_2&0\\0&0 \end{pmatrix} \,  Q^* 
   =P[(e_1e_2^* +e_2e_3^*)\otimes C_2]Q^*
   = p_1C_2q_2^* + p_2C_2 q_3^*,
\end{split}
\end{equation}
 where $\{e_1,e_2,e_3,e_4\}$ is the standard basis for 
 $\C^4.$  Let $T=Q_{4,3}^{-*}$ and $S=P_{1,1}^{-*}.$

\begin{lemma}
 \label{l:PnQ}
 With notations above, $Q_{1,4}$ and $P_{4,4}$ are unitary, $Q_{4,3}$ and $P_{1,1}$ are invertible  and
\begin{equation*}
%\label{e:PnQ+}
 \begin{split}
  P&= \begin{pmatrix} P_{1,1} & P_{1,2}&P_{1,3}&0\\
    0& \ell_{2,2} C_2 T C_2^{-1} & \ell_{2,1} C_2 T C_1^{-1} & b_2 C_2 P_{4,4}\\
    0& \ell_{1,2} C_1 T C_2^{-1}  & \ell_{1,1} C_1 T C_1^{-1} & b_1 C_1 P_{4,4}\\
    0&0&0&P_{4,4} \end{pmatrix} \\
  Q& =\begin{pmatrix} 0 & 0& 0& Q_{1,4} \\
    \ell_{1,1}^* C_1^* S C_1^{-*} & \ell_{1,2}^* C_1^*SC_2^{-*}  & 0& b_1^*C_1^* Q_{1,4}  \\
   \ell_{2,1}^* C_2^* S C_1^{-*} & \ell_{2,2}^* C_2^* SC_2^{-*} & 0 & b_2^*C_2^*  Q_{1,4}\\
   Q_{4,1} & Q_{4,2} & Q_{4,3} & 0 \end{pmatrix}.
 \end{split}
\end{equation*}
\end{lemma}

\begin{proof}
 With
\[
 W_j =\begin{pmatrix} 0&\ell_{1,j}C_2^{-1} &-\ell_{2,j}C_1^{-1} &0 \end{pmatrix}, 
\]
 observe 
\[
 0= W_1 Y_1 = W_1 p_1 C_1q_1^* +W_1p_2 C_1 q_3^*.
\]
Since $Q$ is invertible, there exist $r_1,\dots,r_4\in M_{4s,s}(\C)$  such that
\[
 q_j^*r_k = \delta_{j,k}.
\]
 Hence, from the first identity in equation~\eqref{e:YvPQ-mat}
 and equation~\eqref{d:Y},
\[
 0= W_1p_1 C_1 q_1^*r_1 + W_1 p_2 C_1 q_3^*r_1 
  = W_1p_1C_1 = 
  [\ell_{1,1} C_2^{-1} P_{2,1} -\ell_{2,1}C_1^{-1}P_{3,1}] C_1 
\]
 and therefore $\ell_{1,1} C_2^{-1}P_{2,1}-\ell_{2,1}C_1^{-1} P_{3,1}=0.$
  A similar argument using the second identity in 
 equation~\eqref{e:YvPQ-mat} shows 
\[
 0 = [\ell_{1,2}C_2^{-1}P_{2,1} -\ell_{2,2} C_1^{-1}P_{3,1}] C_2
\]
 and thus
 $\ell_{1,2} C_2^{-1} P_{2,1} -\ell_{2,2}C_1^{-1}P_{3,1}=0.$
 Since $\fL$ from equation~\eqref{d:fL} is invertible, we conclude  $C_2^{-1}P_{2,1}=C_1^{-1}P_{3,1}=0$
 and therefore $P_{2,1}=P_{3,1}=0.$

  Since $[e_4^*\otimes I_s]\, Y_j=0,$ it also follows that 
  $P_{4,1}=P_{4,2}=P_{4,3}=0.$ In particular
  $P_{1,1}$ is invertible. Examining the
  (block) $(4,4)$ entry of $PP^*=B_0$ shows $P_{4,4}P_{4,4}^*=I_s.$
  From the last column  of $PP^*=B_0$ it now
  follows that $P_{1,4}=0,$ and $P_{2,4}=b_2C_2 P_{4,4}$ 
 as well as $P_{3,4}=b_1C_1 P_{4,4}.$ 
  At this point we have identified, as indicated, the
 first and last row and column of $P.$ Similar reasoning
  applies to the third and forth columns and first and fourth
 rows of $Q.$

Comparing the descriptions of $Y_j$ 
 from equations~\eqref{e:YvPQ-mat}.
and \eqref{d:Y} gives,
\begin{equation}
 \label{e:PnQ}
  \begin{split}
    P_{1,1}C_1Q_{2,1}^* &= \ell_{1,1}C_1\\
    P_{1,1} C_1 Q_{3,1}^*&= \ell_{2,1}C_2\\
     P_{1,1}C_2 Q_{2,2}^*&=\ell_{1,2} C_1 \\
    P_{1,1} C_2 Q_{3,2}^*&=\ell_{2,2}C_2 
\end{split}
\quad \quad \quad
%    P_{1,1}C_1Q_{4,1}^* + P_{1,3}C_1^* Q_{4,3}^* &=0\\
\begin{split}
    P_{2,3}C_1Q_{4,3}^* & = \ell_{2,1}C_2 \\    
    P_{3,3} C_1 Q_{4,3}^* &=\ell_{1,1}C_1 \\   
 %   P_{1,1}C_2 Q_{4,2}^* + P_{1,2}C_2^*Q_{4,3}^* &=0\\
    P_{2,2}C_2 Q_{4,3}^* &= \ell_{2,2}C_2 \\
    P_{3,2}C_1 Q_{4,3}^* & = \ell_{1,2}C_2.
  \end{split}
\end{equation}
 Comparing the identities above with the definitions
 of $T$ and $S$ completes the proof.
\end{proof}

Let
\begin{equation*}
%\label{d:wP}
 \Sigma = \begin{pmatrix} P_{2,2} & P_{2,3}\\P_{3,2} & P_{3,3}
  \end{pmatrix}, \ \ 
 \Lambda=\begin{pmatrix} b_2 C_2 P_{4,4} \\ b_1 C_1 P_{4,4} 
   \end{pmatrix}.
\end{equation*}
 From the middle $2\times 2$ block of $PP^*=B_0,$ 
\begin{equation}
 \label{e:Sig-Lam}
 \Sigma\Sigma^* + \Lambda\Lambda^* = I
  =\begin{pmatrix} I &0\\0&I\end{pmatrix}.
\end{equation}
 Let
\[
 G=\begin{pmatrix} \ell_{1,1} & -\ell_{2,1}
    \\ -\ell_{1,2} & \ell_{2,2} \end{pmatrix} 
    \, \begin{pmatrix} C_2^{-1} & 0\\ 0&C_1^{-1}
     \end{pmatrix},
\]
 and observe
\[
%\begin{split}
 G\Sigma %&
 = \det\fL \, \begin{pmatrix} TC_2^{-1} &0\\ 0 & TC_1^{-1}
    \end{pmatrix}, %
 \ \ \ 
 G\Lambda = \begin{pmatrix}\ell_{1,1}b_2 -\ell_{2,1}b_1 \\
             -\ell_{1,2}b_2 +\ell_{2,2}b_1 \end{pmatrix}\, P_{4,4}.
%\end{split}
\]
 Thus, applying $G$ on the left and $G^*$ on the right of
 equation~\eqref{e:Sig-Lam} and comparing the $(1,2)$
 (block) entries gives,
\begin{equation}
\label{e:pres=1}
-[\ell_{1,1}\ell_{1,2}^* C_2^{-1}C_2^{-*}
   + \ell_{2,1} \ell_{2,2}^* C_1^{-1}C_1^{-*}]
 \, = \, (\ell_{1,1}b_2 -\ell_{2,1}b_1)
    (-\ell_{1,2}b_2 +\ell_{2,2}b_1)^*.
\end{equation}

\subsection{A dichotomy}
\label{s:dichotomy}
 With equation~\eqref{e:pres=1} in place, 
 we are now in position to state and prove
 the following  lemma. Recall $s$ is the size of $C.$

\begin{lemma}
 \label{l:s=1}
 If $s>1,$ then $b_1b_2^*=0$ and 
 \begin{enumerate}[(i)]
  \item $\ell_{1,2}=0=\ell_{2,1}$  or 
  \item $\ell_{1,1}=0=\ell_{2,2}.$
 \end{enumerate}
\end{lemma}

\begin{proof}
 If the right hand side of equation~\eqref{e:pres=1} is
 not $0,$ then either $C_1^{-1}C_{1}^{-*}$ and $C_2^{-1}C_2^{-*}$
 commute or $C_j^{-1}C_j^{-*}$ is a multiple of
 the identity for either $j=1$ or $j=2.$  
 In either case  $C_1^*C_1$ and $C_2^*C_2$ commute
 and thus, as $\{C_1^*C_1,C_2^*C_2\}$ generates
 $M_s(\CC)$ as a C-star  algebra, $s=1.$
 Thus, if  $s>1,$ then  the right hand side of
  equation~\eqref{e:pres=1}  is $0.$

  If  $\ell_{1,1}\ell_{1,2}^*\ne0$ or $\ell_{2,1} \ell_{2,2}^*\ne 0,$
 then they are both not zero and again, $C_1^*C_1$ and $C_2^*C_2$
 commute and $s=1.$ Thus, if $s>1,$ then
  $\ell_{1,1}\ell_{1,2}^* = 0 =\ell_{2,1} \ell_{2,2}^*.$
 Since $\fL$ is invertible, there are two cases,
 either $\ell_{2,1}=0=\ell_{1,2}$ (and $\ell_{1,1}\ne 0\ne \ell_{2,2}$)
 or $\ell_{1,1} =  0 = \ell_{2,2}$ (and 
 $\ell_{2,1} \ne 0 \ne \ell_{1,2}$). In either case $b_1b_2^*=0.$
\end{proof}

\begin{lemma}
\label{l:s>1:b=0}
 If $s>1,$ then $b=0.$
\end{lemma}

\begin{proof}
 Arguing by contradiction, 
 suppose $s>1$ and $b\ne 0.$  Observe that conclusion that $b_1b_2^*=0$ of  Lemma~\ref{l:s=1} applies to any automorphism of $\fP.$   Thus, without loss of generality,  $b_1\ne 0$ and $b_2=0.$  Another appeal to  Lemma~\ref{l:s=1} gives either $\ell_{j,k}=0$ for $j\ne k$ or $\ell_{j,k}=0$ for $j=k.$ Suppose the first case holds. Given $\theta$ real, let  $f_\theta$ denote the automorphism of $f_\theta$ given by $f_\theta(y_1,y_2)=e^{i\theta}(y_2,y_1)$ and let $\psi=\varphi\circ f_\theta.$ Thus $\psi:\fP\to\fP$ is an automorphism  and 
\[
 \psi(0)=\varphi\circ f_\theta(b_1,0) = \varphi(0,e^{i\theta}b_1) = (\varphi_1(0,e^{i\theta}b_1),\, \varphi_2(0,e^{i\theta}b_1)).
\]
 Hence $\varphi_1(0,e^{i\theta}b_1) \, \varphi_2(0,e^{i\theta}b_1)$ is identically $0.$
 By analyticity,  $\varphi_k(0,e^{i\theta}b_1)=0$ for some $k$ and all $\theta$ and hence, for either $k=1$ or $k=2,$ we have $g_k(z)=\varphi_k(0,z)$ is identically $0.$ Since $g_1(0)=b_1\ne 0,$ it follows that $g_2(z)$ is identically $0.$ Thus $0=g_2^\prime(0)=\ell_{2,2}\ne 0,$ a contradiction. 

Now suppose instead that $\varphi_{j,j}=0$ for $j=1,2.$ In this case set $f_\theta(y_1,y_2)=e^{i\theta}(y_1,y_2)$ and $\psi=\varphi\circ f_\theta\circ\varphi.$ Thus, 
\[
 \psi(0)= \varphi(e^{i\theta}b_1,0)
\]
 and $\varphi_1(e^{i\theta}b_1,0) \, \varphi_2(e^{i\theta}b_1,0)$ is identically $0.$ Hence
 $g_2(z)=\varphi_2(z,0)$ is identically $0$ and thus $0=g_2^\prime(0)= \ell_{2,1}\ne 0,$ a contradiction which shows $b_1=0=b_2.$
\end{proof}

\begin{lemma}
\label{l:s=1:b=0}
 If $s=1,$ then $b=0.$ 
\end{lemma}

\begin{proof}
 Since $s=1,$ we have $C_1,C_2\in \C$ are unimodular 
 and  $\fP[1]$ is the set $\{z=(z_1,z_2)\in \CC^2: |z_1|+|z_2|<1\}.$
 Indeed, given $(z_1,z_2)\in \CC^2,$  the matrix
\[
  \begin{pmatrix} z_1^* & z_2\\z_2^* & z_1 \end{pmatrix}
\] 
 is a contraction if and only if the self-adjoing
 matrix,
\[
  \begin{pmatrix} r & z_2e^{it} \\ z_2^*e^{-it} & r \end{pmatrix}
\]
 is a contraction, where $z_1=re^{it}$ is the polar
 decomposition of $z_1.$  This latter matrix
 has eigenvalues $r\pm |z_2|$ and hence is a 
 contraction if and only if $|z_1|+|z_2|<1.$
 
 The set $\fP[1]$ is known as a \emph{psuedo-ellipse} and 
 $f=\varphi[1]:\fP[1]\to\fP[1]$ is an automorphism (in the
  classical  several complex variables sense). It  is known 
 (see \cite{first-steps}) that automorphisms of $\fP[1]$  
 are compositions of maps
 of the form $(z_1,z_2)\mapsto (\gamma_1 z_1,\gamma_2 z_2)$
 and $(z_1,z_2)\mapsto (z_2,z_1),$ where $\gamma_j\in \CC$
 are unimodular (the proof uses techniques from several
 complex variables and lie groups).  In particular,
 $b=0.$   Thus in any case $b=0$ and $\varphi$ 
 is linear by \cite[Theorem~4.4]{proper},
 since $\varphi(0)=b=0$ and the domain $\fP$ 
 is circularly symmetric. 
\end{proof}

 We now give a self contained alternate proof of
 Lemma~\ref{l:s=1:b=0} based upon results in this article.

\begin{proof}[Second proof of Lemma~\ref{l:s=1:b=0}]
  In addition to the identities
 of equation~\eqref{e:PnQ} obtained by
 comparing the equations~\eqref{e:YvPQ-mat} 
 and \eqref{d:Y}, observe
\begin{equation*} 
\begin{split}
 P_{1,1} Q_{4,1}^* + P_{1,3} Q_{4,3}^* &=0\\
     P_{1,1} Q_{4,2}^* + P_{1,2} Q_{4,3}^* &=0,
\end{split}
\end{equation*}
which can be summarized as
\begin{equation}
 \label{e:14entry}
 Q_{4,3}^* \begin{pmatrix} P_{1,3} \\ P_{1,2}\end{pmatrix}
  = - P_{1,1} \begin{pmatrix} Q_{4,1}^* \\ Q_{4,2}^*\end{pmatrix}.
\end{equation}

 From $(1,2)$ and $(1,3)$ entries of $PP^*=B_0$
 and using $T_j=Q_{4,3}^{-*},$ 
\[
 \begin{pmatrix} P_{1,2}& P_{1,3} \end{pmatrix}\, 
 \begin{pmatrix} \ell_{2,2}^* & \ell_{1,2}^*\\ 
     \ell_{2,1}^*&\ell_{1,1}^* \end{pmatrix}
  = Q_{4,3} \begin{pmatrix} b_1 & b_2 \end{pmatrix}.
\]
 Equivalently,
\[
  Q_{4,3}^{-*}
  \begin{pmatrix} P_{1,3}^*&P_{1,2}^*\end{pmatrix}
 \, \begin{pmatrix} \ell_{1,1} & \ell_{2,1} \\ 
    \ell_{1,2} & \ell_{2,2} \end{pmatrix}
  = \begin{pmatrix} b_2^* & b_1^* \end{pmatrix}.
\]
 Similarly from  the $(4,2)$ and $(4,3)$ entries of $QQ^*=B_0,$
\[
  P_{1,1}^{-1} \, \begin{pmatrix} Q_{4,1} & Q_{4,2} \end{pmatrix}
 \, \begin{pmatrix} \ell_{1,1} & \ell_{2,1} \\ 
    \ell_{1,2} & \ell_{2,2} \end{pmatrix} 
  = \begin{pmatrix} b_2^* & b_1^* \end{pmatrix}.
\]
  It follows that
\begin{equation}
\label{e:14entry-alt}
 Q_{4,3}^{-1} \begin{pmatrix} P_{1,3} \\ P_{1,2}\end{pmatrix}
  = P_{1,1}^{-*} \begin{pmatrix} Q_{4,1}^* \\ Q_{4,2} \end{pmatrix}.
\end{equation}
 From equations~\eqref{e:14entry} and \eqref{e:14entry-alt},
\[
 (|Q_{4,3}|^2+|P_{1,1}|^2)
    \begin{pmatrix} Q_{4,1}^* \\ Q_{4,2} \end{pmatrix} =0
\]
 and we conclude  $Q_{4,1}=Q_{4,2}=0$ and thus $b=0.$
\end{proof}

\subsection{Completion of the proof of Theorem~\ref{t:main}}
\label{s:concludes}
 From Lemmas~\ref{l:s>1:b=0} and Lemma~\ref{l:s=1:b=0} it follows
 that $b=0$ and hence, by \cite[Theorem~4.4]{proper}, $\varphi$
 is linear. 

 Since $\varphi$ is linear,
\[
 \varphi(x) = (w_1, w_2),
\]
where
\[
  \begin{pmatrix} w_1 \\ w_2 \end{pmatrix}
  =\fL \begin{pmatrix} x_1 \\ x_2 \end{pmatrix}.
\]

 If $s>1,$ then, by Lemma~\ref{l:s=1} and composing with the
 automorphism $(x_1,x_2)\mapsto (x_2,x_1),$ we assume $\fL$ is 
 diagonal, in which case the diagonal entries must be unimodular
 and the proof is complete.

 Now suppose $s=1.$ From the (first) proof of
  Lemma~\ref{l:s=1:b=0} it can be seen that $\varphi$ is trivial.
 Alternately, from Lemma~\ref{l:PnQ} and the relation $PP^*=B_0,$
  it follows that the matrix
\[
 \begin{pmatrix} \ell_{2,2} & \ell_{2,1} \\ \ell_{1,2} & \ell_{1,1} \end{pmatrix}\, T 
\]
 is unitary. Hence $\fL$ is a multiple of a unitary. Since $\varphi$ is an automorphism of $\fP,$ we conclude that $\fL$ is unitary. 
 In particular, $|\ell_{1,1}|^2+|\ell_{2,1}|^2=1.$ 
 Since $(t,0)$ is in  $\fP$ for $0<t<1,$  so is $\varphi(t,0)=t(\ell_{1,1},\ell_{2,1}).$ It now follows from  
item~\ref{i:fP-alt3} of Proposition~\ref{p:fP-alt} that $\ell_{1,1}\ell_{2,1}=0.$ A similar argument shows $\ell_{1,2}\ell_{2,2}=0.$ Hence $\cL$ is unitary and either $\ell_{j,k}=0$ for $j\ne k$  
or $\ell_{j,k}=0$ for $j=k$ and the proof is complete.
\qed

\newpage

\printindex

\end{document}